\documentclass[a4paper,12pt]{amsart}
\usepackage{amssymb}
\usepackage{ifthen}
\usepackage[usenames]{color}
\usepackage{graphicx}
\usepackage[active]{srcltx}
 \makeatletter
\renewcommand*\subjclass[2][2000]{%
  \def\@subjclass{#2}%
  \@ifundefined{subjclassname@#1}{%
    \ClassWarning{\@classname}{Unknown edition (#1) of Mathematics
      Subject Classification; using '1991'.}%
  }{%
    \@xp\let\@xp\subjclassname\csname subjclassname@#1\endcsname
  }%
}
 \makeatother
\usepackage{enumerate,url,amssymb,  mathrsfs}%, pdfsync}% ,showkeys, pdfsync}

\newtheorem{theorem}{Theorem}[section]

\newtheorem{case}{Case}[section]
\newtheorem{lemma}{Lemma}[section]

\theoremstyle{definition}

\theoremstyle{remark}
\newtheorem{remark}{Remark}[section]

\numberwithin{equation}{section}

\newcounter{alphabet}
\newcounter{tmp}

\makeatletter
\newcommand{\Ref}[1]{\@ifundefined{r@#1}{}{\setcounter{tmp}{\ref{#1}}\Alph{tmp}}}
\makeatother
\def\XXint#1#2#3{{\setbox0=\hbox{$#1{#2#3}{\int}$}
\vcenter{\hbox{$#2#3$}}\kern-.5\wd0}}

{}

\begin{document}

\title{Optimal estimates for hyperbolic harmonic mappings in Hardy space}

\author{Jiaolong Chen}
\address{Jiaolong Chen, Key Laboratory of High Performance Computing and Stochastic Information Processing (HPCSIP)
(Ministry of Education of China), School of Mathematics and Statistics, Hunan Normal University, Changsha, Hunan 410081, People's Repulic of China}
\email{jiaolongchen@sina.com}

\author[David Kalaj]{David Kalaj}
\address{University of Montenegro, Faculty of Natural Sciences and
Mathematics, Cetinjski put b.b. 81000 Podgorica, Montenegro}
\email{davidkalaj@gmail.com}



\maketitle

%\def\thefootnote{}
%\footnotetext{
%\texttt{\tiny File:~\jobname .tex,
%          printed: \number\year-\number\month-\number\day,
%          \thehours.\ifnum\theminutes<10{0}\fi\theminutes}
%}
\makeatletter\def\thefootnote{\@arabic\c@footnote}\makeatother

\begin{abstract}
Assume that $p\in(1,\infty]$ and $u=P_{h}[\phi]$,
where $\phi\in L^{p}(\mathbb{S}^{n-1},\mathbb{R}^{n})$.
Then for any $x\in \mathbb{B}^{n}$,
we obtain the sharp inequalities
$$
|u(x)|\leq
 \frac{\mathbf{C}_{q}^{\frac{1}{q}}(x)}{(1-|x|^2)^{\frac{n-1 }{p}}}
\|\phi\|_{L^{p}}
\quad\text{and}\quad
|u(x)|\leq
 \frac{\mathbf{C}_{q}^{\frac{1}{q}} }{(1-|x|^2)^{\frac{n-1 }{p}}}
\|\phi\|_{L^{p} }
$$
for some function $\mathbf{C}_{q}(x)$ and constant $\mathbf{C}_{q}$ in terms of Gauss hypergeometric and Gamma functions,
where $q$ is the conjugate of $p$.
This result generalize and extend some known result from harmonic mapping theory (\cite[Theorems 1.1 and 1.2]{kama} and
\cite[Proposition 6.16]{ABR}).
\end{abstract}

\maketitle
%\tableofcontents

\section{Introduction}\label{intsec}
 For $n\geq2$, let $\mathbb{R}^{n}$ denote the $n$-dimensional Euclidean space.
We use $\mathbb{B}^{n}$ and $\mathbb{S}^{n-1}$ to denote the unit ball $\{x\in\mathbb{R}^{n}:|x|<1\}$ and the unit sphere $\{x\in\mathbb{R}^{n}:|x|=1\}$, respectively.
In particular, let
 $\overline{\mathbb{B}}^{n}=\mathbb{B}^{n} \cup \mathbb{S}^{n-1}$,
$\mathbb{R}^{2}=\mathbb{C}$
 and $\mathbb{B}^{2}=\mathbb{D}$.

A mapping $u=(u_1,\cdots,u_n)\in C^{2}(\mathbb{B}^{n}, \mathbb{R}^{n})$ is said to be {\it hyperbolic harmonic} if
 $$\Delta_{h}u=(\Delta_{h}u_{1}, \cdots,\Delta_{h}u_{n})=0,$$  that is,
 for each $j\in \{1,\cdots, n\}$, $u_j$ satisfies the hyperbolic Laplace equation
$$
\Delta_{h}u_{j} (x)=(1-|x|^2)^2\Delta u_{j}(x)+2(n-2)(1-|x|^2)\sum_{i=1}^{n} x_{i} \frac{\partial u_{j}}{\partial x_{i}}(x)=0,
$$
 where
 $\Delta$ denotes the usual Laplacian in $\mathbb{R}^{n}$.
For convenience, in the rest of this paper, we call $\Delta_{h}$ the {\it hyperbolic Laplacian operator}.

Obviously, when $n=2$, $\Delta_{h}u=(1-|x|^2)^2\Delta u$, and thus the class of hyperbolic harmonic mappings coincides with the usual class of harmonic mappings in $\mathbb{D}$.
 However, when $n\geq 3$, it is easily seen that the only
mappings annihilated by both $\Delta_{h}$ and $\Delta$ are the constant mappings (cf. \cite{sto2012}).
In this paper, we focus our investigations on the case when $n\geq 3$.

\subsection{Hardy space for hyperbolic harmonic mappings}
A measurable mapping $f: \mathbb{B}^{n}\rightarrow\mathbb{ R}^{n}$ belongs to the {\it Hardy space} $\mathcal{H}^{p}(\mathbb{B}^{n}, \mathbb{ R}^{n})$ with $p\in(0,\infty]$, if
 $M_{p}(r,f)$ exists for all $r\in (0,1)$ and $||f||_{\mathcal{H}^p}<\infty$, where
$$||f||_{\mathcal{H}^p}=\sup_{0<r<1} \big\{M_{p}(r,f)\big\}$$
and
$$\;\;M_{p}(r,f)=\begin{cases}
\displaystyle \;\left( \int_{\mathbb{S}^{n-1}} |f(r\xi)|^{p}d\sigma(\xi)\right)^{\frac{1}{p}} , & \text{ if } p\in (0,\infty),\\
\displaystyle \;\sup_{\xi\in \mathbb{S}^{n-1}} \big\{|f(r\xi)|\big\} , \;\;\;\;& \text{ if } p=\infty.
\end{cases}$$

\noindent Here and hereafter, $d \sigma$ always denotes the normalized surface measure on $\mathbb{S}^{n-1}$ so that $\sigma(\mathbb{S}^{n-1})=1$.

If $\phi\in L^{1}(\mathbb{S}^{n-1},\mathbb{R}^{n})$, we define the {\it invariant Poisson integral} or {\it Poisson-Szeg\"{o} integral} of $\phi$ in $\mathbb{B}^{n}$ by
\begin{eqnarray}\label{eq-1.1}
P_{h}[\phi](x)=\int_{\mathbb{S}} P_{h}(x,\zeta)\phi(\zeta)d\sigma(\zeta)
\end{eqnarray}
 (cf. \cite{sto2012} or \cite[Definition 5.3.2]{sto2016}), where
\begin{eqnarray}\label{eq-1.2}
 P_{h}(x,\zeta)=\left(\frac{1-|x|^2}{|x-\zeta|^{2}}\right)^{n-1}
\end{eqnarray}
is the {\it Poisson-Szeg\"{o} kernel} with respective to $\Delta_{h}$ satisfying
\begin{eqnarray}\label{eq-1.3}
\int_{\mathbb{S}} P_{h}(x,\zeta) d\sigma(\zeta)=1
\end{eqnarray}
(cf. \cite[Lemma 5.3.1(c)]{sto2016}).
Similarly, if $\mu$ is a finite signed Borel measure on $\mathbb{S}^{n-1}$,
then invariant Poisson integral of $\mu$ will be denoted by $P_{h}[\mu]$, that is,
$$
P_{h}[\mu](x)=\int_{\mathbb{S}} P_{h}(x,\zeta) d\mu(\zeta).
$$
Furthermore, both $P_{h}[\phi]$ and $P_{h}[\mu]$ are hyperbolic harmonic in
$\mathbb{B}^{n}$ (cf. \cite{bur, sto2016}).

It is  known that if $u$ is
a hyperbolic harmonic mapping and
$u\in \mathcal{H}^p(\mathbb{B}^n, \mathbb{R}^{n})$ with $p\in(1,\infty]$,
 then $u$ has the following integral representation
 (cf. \cite[Theorem 7.1.1(c)]{sto2016})
$$
u(x)= P_{h}[\phi](x),
$$
where
$\phi\in L^{p}(\mathbb{S}^{n-1},\mathbb{R}^{n})$
is the boundary value of $u$ and
$$
\|\phi\|_{L^p}= \|u\|_{\mathcal{H}^p}.
$$
If $\Delta_{h}u=0$ and $u\in \mathcal{H}^1(\mathbb{B}^n, \mathbb{R})$, then $u$ has the representation $u =P_{h}[\mu]$,
where $\mu$ is a signed Borel measure in $\mathbb{B}^n$.
Further, the similar arguments as \cite[Page 118]{ABR} show that
$\|u\|_{\mathcal{H}^1}=||\mu||$, where $||\mu||$ is the total variation of $\mu$ on $\mathbb{S}^{n-1}$.

Let $H^{p}(\mathbb{D},\mathbb{C})$  denote the Hardy space which consists of analytic functions from $\mathbb{D}$ into $\mathbb{C}$, while the analogous space of harmonic mappings from
$\mathbb{B}^{n}$ into $\mathbb{R}^{n}$ is denoted by $h^{p}(\mathbb{B}^{n}, \mathbb{ R}^{n})$.
In \cite[Lemma 5.1.1]{pav}, Pavlovi\'{c} obtained a growth estimate
for functions in $H^{p}(\mathbb{D},\mathbb{C})$:
If $f\in H^{p}(\mathbb{D},\mathbb{C})$ with $p\in(0,\infty]$,
then for any $x\in \mathbb{D}$,
 $$
 |f(z)|\leq (1-|z|^{2})^{-\frac{1}{p}}||f||_{H^{p}}.
 $$
For the harmonic case, by \cite[Proposition 6.16]{ABR}, we see that for any $f\in h^{p}(\mathbb{B}^{n}, \mathbb{ R}^{n})$ with $p\in[1,\infty)$,
$$
|f(x)|\leq
\left( \frac{1+|x|}{(1-|x|)^{n-1}}\right)^{ \frac{1}{p}}
||f||_{h^{p}}
$$
in $\mathbb{B}^{n}$.
Further, if $p=2$, then $h^{2}(\mathbb{B}^{n}, \mathbb{ R}^{n})$ is a Hilbert space.
In this case, we have the following slightly better point estimate (cf. \cite[Proposition 6.23]{ABR})
$$
|f(x)|\leq
\sqrt{ \frac{1+|x|^{2}}{(1-|x|^{2})^{n-1}}} \;
||f||_{h^{2}}.
$$
In the recent paper \cite{kama},  Kalaj and Markovi\'{c}
studied the pointwise estimates of mappings in $h^{p}(\mathbb{B}^{n}, \mathbb{ R}^{n})$, where $p\in(1,\infty]$.
They obtained a sharp function
 $\mathcal{C}_{p}(x)$ and a sharp constant $\mathcal{C}_{p}$ for the following two inequalities
$$
|u(x)|\leq  \frac{\mathcal{C}_{p}(x)}{(1-|x|^{2})^{\frac{n-1}{p}}}
||u||_{h^{p}}
\quad\text{and}\quad
|u(x)|\leq  \frac{\mathcal{C}_{p}}{(1-|x|^{2})^{\frac{n-1}{p}}}
||u||_{h^{p}}.
$$

 In this paper, we will establish a counterpart of
 \cite[Theorems 1.1 and 1.2]{kama} in the setting of hyperbolic harmonic mappings in $\mathcal{H}^p(\mathbb{B}^{n},\mathbb{R}^{n})$.
In order to state our main results, we need to introduce some notations.

For any $a\in\mathbb{R}$ and $k\in \mathbb{N}=\{0,1,2,\ldots\}$,
let $(a)_{k}$ denote the {\it factorial function} with $(a)_{0}=1$
and $(a)_{k}=a(a+1)\ldots (a+k-1)$.
For  $x\in \mathbb{R}$,
we define the {\it Gauss hypergeometric function} or {\it hypergeometric series} by
$$
\;_{2}F_{1}(a,b;c;x)
=\sum_{k=0}^{\infty}\frac{(a)_{k}(b)_{k}}{k!(c)_{k}}x^{k},
$$
where $a,b\in \mathbb{R}$  and $c$ is neither zero nor a negative integer (cf. \cite[Page 46]{rain}).
If $c-a-b>0$, then the series $\;_{2}F_{1}(a,b;c;x)$ converges absolutely for all $|x|\leq 1$ (cf. \cite[Section 31]{rain}).
If $c>b>0$ and $|x|<1$, then
(cf. \cite[Page 47]{rain})
\begin{eqnarray}\label{eq-1.4}
\;_{2}F_{1}(a,b;c;x)
= \frac{\Gamma(c)}{\Gamma(b)\Gamma(c-b)}
\int_{0}^{1}\frac{t^{b-1}(1-t)^{c-b-1}}{(1-tx)^{a}}dt,
\end{eqnarray}
where $\Gamma$ is the Gamma function.
It is easy to check the following formula
\begin{eqnarray}\label{eq-1.5}
\frac{d}{dx}\;_{2}F_{1}(a,b;c;x)
= \frac{ab}{c}
 \;_{2}F_{1}(a+1,b+1;c+1;x).
\end{eqnarray}
If $2b$ is neither zero nor a negative integer, and if both $|x|<1$ and
$|4x(1+x)^{-2}|<1$, then we have the following quadratic transformation (cf. \cite[Page 65]{rain})
\begin{eqnarray}\label{eq-1.6}
\;_{2}F_{1}\left(a,b;2b;\frac{4x}{(1+x)^{2}}\right)
= (1+x)^{2a}
\;_{2}F_{1}\left(a,a-b+\frac{1}{2};b+\frac{1}{2};x^{2}\right).
\end{eqnarray}

The following are our results.
Note that a different form of growth estimate
for hyperbolic harmonic mappings in $\mathcal{H}^{p}(\mathbb{B}^{n},\mathbb{R}^{n})$
was proved in \cite[Theorem 1.1]{chenkalaj}.

\begin{theorem}\label{thm-1.1}
Let $p\in(1,\infty]$ and $q$ be its conjugate.
If $u=P_{h}[\phi]$ and $\phi\in L^{p}(\mathbb{S}^{n-1},\mathbb{R}^{n})$,
then for any $x\in\mathbb{B}^{n}$, we have the following sharp
inequality:
\begin{eqnarray*}
|u(x)|\leq
 \frac{\mathbf{C}_{q}^{\frac{1}{q}}(x)}{(1-|x|^2)^{\frac{n-1 }{p}}}
\|\phi\|_{L^{p}},
\end{eqnarray*}
 where
$$
\mathbf{C}_{q}(x)
=
\;_{2}F_{1}\left(-(n-1)(q-1),\frac{n}{2}+q-nq;\frac{n}{2};|x|^{2}\right).
$$

\end{theorem}

\begin{theorem}\label{thm-1.2}
Let $p\in(1,\infty]$ and $q$ be its conjugate.
If $u=P_{h}[\phi]$ and $\phi\in L^{p}(\mathbb{S}^{n-1},\mathbb{R}^{n})$,
then for any $x\in\mathbb{B}^{n}$, we have the following sharp
inequality:
\begin{eqnarray*}
|u(x)|\leq
 \frac{\mathbf{C}_{q}^{\frac{1}{q}} }{(1-|x|^2)^{\frac{n-1 }{p}}}
\|\phi\|_{L^{p}},
\end{eqnarray*}
 where
$$
\mathbf{C}_{q}
=
\;_{2}F_{1}\left(-(n-1)(q-1),\frac{n}{2}+q-nq;\frac{n}{2};1\right).
$$
\end{theorem}

\begin{remark}\label{re-1.1}
If $u=P_{h}[\phi]$ and $\phi\in L^{\infty}(\mathbb{S}^{n-1},\mathbb{R}^{n})$, then by \eqref{eq-1.1} and \eqref{eq-1.3}, it is easy to see that for any $x\in \mathbb{B}^{n}$,
\begin{eqnarray}\label{eq-1.7}
|u(x)|\leq
\int_{\mathbb{S}^{n-1}}P_{h}(x,\zeta)d\sigma(\zeta)\cdot\|\phi\|_{L^{\infty}}
=  \|\phi\|_{L^{\infty}}.
\end{eqnarray}
By letting $\phi\equiv C$ on $\mathbb{S}^{n-1}$, we see that  the sharpness of \eqref{eq-1.7} follows,
where $C$ is a constant.
In this case, $p=\infty$, $q=1$ and $\mathbf{C}_{1}(x)\equiv\mathbf{C}_{1}\equiv1$.

If $u=P_{h}[\phi]$ and $\phi\in L^{1}(\mathbb{S}^{n-1},\mathbb{R}^{n})$, then by \eqref{eq-1.1} and \eqref{eq-1.2}, we obtain
\begin{eqnarray}\label{eq-1.8}
|u(x)|\leq \max_{\zeta\in \mathbb{S}^{n-1}}P_{h}(x,\zeta)\cdot\|\phi\|_{L^{1}}
= \frac{(1+|x|)^{2(n-1)} }{( 1-|x|^{2} )^{n-1}} \|\phi\|_{L^{1}}
\end{eqnarray}
in $\mathbb{B}^{n}$.
In the following, we show the sharpness of \eqref{eq-1.8}.
For any $x_{0}=|x_{0}|\eta_{0} \in \mathbb{B}^{n}$ and $i\in \mathbb{Z}^{+}$,
define
$$
\phi_{i}(\zeta)=\frac{ \chi_{\Omega_{i}}(\zeta)}{ || \chi_{\Omega_{i}} ||_{L^{1}}}
$$
on $\mathbb{S}^{n-1}$
and
$u_{i} =P_{h}[\phi_{i}] $ in $\mathbb{B}^{n}$,
where
$\Omega_{i}=\{\zeta\in\mathbb{S}^{n-1}:|\zeta-\eta_{0}|\leq\frac{1}{i}\}$
and $\chi$ is the indicator function.
Then for any $i\in \mathbb{Z}^{+}$ and $x\in \mathbb{B}^{n}$,
$ ||\phi_{i}||_{L^{1} }=1$ and
\begin{eqnarray} \label{eq-1.9}
 u_{i}(x)
 = \int_{\mathbb{S}^{n-1}}
 P_{h}(x,\zeta)
 \frac{ \chi_{\Omega_{i}}(\zeta)}{ || \chi_{\Omega_{i}} ||_{L^{1}}}
  d\sigma(\zeta).
  \end{eqnarray}
For fixed $x\in \mathbb{B}^{n}$,
by the definition of $\chi_{\Omega_{i}}$, we see that
$$
\lim_{i\rightarrow\infty}
\big| P_{h}(x,\zeta)  -
P_{h}(x,\eta_{0}) \big|\cdot
 \chi_{\Omega_{i}}(\zeta)
 =0.
$$
Then for any $\varepsilon>0$, there exists a positive integer $m_{1}=m_{1}(\varepsilon,x,\eta_{0})$ such that for any $i\geq m_{1} $,
$$
\big| P_{h}(x,\zeta)  -
P_{h}(x,\eta_{0}) \big|\cdot
 \chi_{\Omega_{i}}(\zeta)<\varepsilon,
$$
where the notation $m_{1}=m_{1}(\varepsilon,x, \eta_{0})$ means that the constant $m_{1}$ depends only on the quantities $\varepsilon$, $x$ and $\eta_{0}$.
Since
$\int_{\mathbb{S}^{n-1}}
\frac{ \chi_{\Omega_{i}}(\zeta)}{||\chi_{\Omega_{i}}||_{L^{1}}}
d\sigma(\zeta)=1$, then for any $i\geq m_{1} $,
\begin{eqnarray*}
&\;\;&\left| \int_{\mathbb{S}^{n-1}}
 P_{h}(x,\zeta)
 \frac{ \chi_{\Omega_{i}}(\zeta)}{ || \chi_{\Omega_{i}} ||_{L^{1}}}
  d\sigma(\zeta)
 -  P_{h}(x,\eta_{0}) \right|\\
&\leq&\int_{\mathbb{S}^{n-1}}
  \left|\big( P_{h}(x,\zeta)  -
P_{h}(x,\eta_{0})\big)\cdot\chi_{\Omega_{i}}(\zeta)\right|
 \cdot
 \frac{ \chi_{\Omega_{i}}(\zeta)}{ || \chi_{\Omega_{i}} ||_{L^{1}}}d\sigma(\zeta)
 \leq\varepsilon.
\end{eqnarray*}
This, together with \eqref{eq-1.9}, means
\begin{eqnarray} \label{eq-1.10}
\lim_{i\rightarrow\infty} u_{i}(x)
=\lim_{i\rightarrow\infty} \int_{\mathbb{S}^{n-1}}
 P_{h}(x,\zeta)
 \frac{ \chi_{\Omega_{i}}(\zeta)}{ || \chi_{\Omega_{i}} ||_{L^{1}}}
  d\sigma(\zeta)
 = P_{h}(x,\eta_{0}).
\end{eqnarray}
By replacing $x$ with $x_{0}$ in \eqref{eq-1.10}, we obtain
$$
\lim_{i\rightarrow\infty} u_{i}(x_{0})
 = P_{h}(x_{0},\eta_{0})
 = \frac{(1+|x_{0}|)^{2(n-1)} }{( 1-|x_{0}|^{2} )^{n-1}} \lim_{i\rightarrow\infty}
 \| \phi_{i}\|_{L^{1}},
$$
and so, \eqref{eq-1.8} is sharp.
\end{remark}

\section{Proofs of the main results}\label{sec-2}

The aim of this section is to prove Theorems \ref{thm-1.1} and \ref{thm-1.2} when $p\in(1,\infty)$.
The cases $p=1$ and $p=\infty$ are already considered in Remark \ref{re-1.1}.
Before the proofs of them, we need some preparation which consists of three lemmas.
The first one reads as follows.

\begin{lemma}\label{lem-2.1}
For any $q\in(1,\infty)$ and $x\in\mathbb{B}^{n}$,
we have the following sharp inequality
\begin{eqnarray*}
|u(x)|
 \leq
 \frac{1}{(1-|x|^2)^{\frac{(n-1)(q-1)}{q}}}
\left( \int_{\mathbb{S}^{n-1}}  |x-\eta|^{2(n-1)(q-1)}   d\sigma(\eta)\right)^{\frac{1}{q}}
\|\phi\|_{L^{p}}.
\end{eqnarray*}
\end{lemma}

\begin{proof}
Let $p$ be the conjugate of $q$, $ \phi \in L^{p}(\mathbb{S}^{n-1},\mathbb{R}^{n})$
 and $u=P_{h}[\phi]$ in $\mathbb{B}^{n}$,
 where
$p\in(1,\infty)$.
For any $x\in\mathbb{B}^{n}$, by \eqref{eq-1.1} and H\"{o}lder's inequality, we have
 \begin{eqnarray}\label{eq-2.1}
|u(x)|
 \leq
 \left(\int_{\mathbb{S}^{n-1}}
 P^{ q}_{h}(x,\zeta)d\sigma(\zeta)\right)^{\frac{1}{q}}
\|\phi\|_{L^{p}}.
\end{eqnarray}

Next we show the sharpness of \eqref{eq-2.1}.
For any $x\in \mathbb{B}^{n}$,
define
$\phi_{*}(\zeta)
= P_{h}^{q/p}(x,\zeta)$
on $\mathbb{S}^{n-1}$ and $u_{*} = P_{h}[\phi_{*}] $ in $\mathbb{B}^{n}$.
 Then we have
$$
   u_{*}(x)
   = \left(\int_{\mathbb{S}^{n-1}} P^{q}_{h}(x,\zeta) d\sigma(\zeta)\right)^{\frac{1}{q}}\|\phi_{*}\|_{L^{p}},
$$
 which means that \eqref{eq-2.1} is sharp for any $x\in \mathbb{B}^{n}$.

In the following, we will calculate the integral above.
 For any $\eta\in \mathbb{S}^{n-1}$ and $x\in \mathbb{B}^{n}$,
let $\zeta=T_{x}(\eta)$, where
$$T_{x}(\eta)=x-(1-|x|^{2})\frac{\eta-x }{|\eta-x|^{2}}.$$
Then $\zeta=T_{x}(\eta)$ is a transformation from
  $\mathbb{S}^{n-1}$ onto $\mathbb{S}^{n-1}$
  (cf. \cite[Section 2.5]{chen2018}) satisfying
\begin{eqnarray}\label{eq-2.2}
 |x-\zeta|=\frac{1-|x|^{2}}{|\eta-x|}
\quad\text{and}\quad
d\sigma(\zeta)=\frac{(1-|x|^2)^{n-1}}{|\eta-x|^{2n-2}}d\sigma(\eta)
 \end{eqnarray}
 (cf. \cite[Page 250]{mark}).
Combining  \eqref{eq-1.2} and \eqref{eq-2.2}, we get
 $$
 P_{h}^{q}(x,\zeta)d\sigma(\zeta)
=\frac{(1-|x|^2)^{(n-1)q}}{|x-\zeta|^{2(n-1)q}}d\sigma(\zeta)
= \frac{|x-\eta|^{2(n-1)(q-1)} }{(1-|x|^2)^{(n-1)(q-1)}} d\sigma(\eta).
$$
Therefore, for any  $q\in(1,\infty)$ and $x\in\mathbb{B}^{n}$ ,
$$
\int_{\mathbb{S}^{n-1}}P_{h}^{q}(x,\zeta)d\sigma(\zeta)
=\frac{1}{(1-|x|^2)^{(n-1)(q-1)}}
 \int_{\mathbb{S}^{n-1}}  |x-\eta|^{2(n-1)(q-1)}   d\sigma(\eta).
$$
 From this and \eqref{eq-2.1}, the lemma follows.
\end{proof}

For any $q\in(1,\infty)$ and $x\in\overline{\mathbb{B}}^{n}$,
let
\begin{eqnarray}\label{eq-2.3}
C_{q}(x)
= \int_{\mathbb{S}^{n-1}}  |x-\eta|^{2(n-1)(q-1)}   d\sigma(\eta).
\end{eqnarray}
For $C_{q}(x)$, we have the following result.
\begin{lemma}\label{lem-2.2}
For any $q\in(1,\infty)$ and $x\in\mathbb{B}^{n}$,
we have
\begin{eqnarray*}
C_{q}(x)
=\;_{2}F_{1}\left(-(n-1)(q-1),\frac{n}{2}+q-nq;\frac{n}{2};|x|^{2}\right).
\end{eqnarray*}
\end{lemma}

\begin{proof}
Let $A$ be an unitary transformation in $\mathbb{R}^{n}$.
For any $x\in\mathbb{B}^{n}$,
by replacing $\eta$ with $A\eta$ in \eqref{eq-2.3}, we get
\begin{eqnarray*}
C_{q}(Ax)
= \int_{\mathbb{S}^{n-1}}  |Ax-A\eta|^{2(n-1)(q-1)}   d\sigma(A\eta)
= \int_{\mathbb{S}^{n-1}}  |x-\eta|^{2(n-1)(q-1)}   d\sigma(\eta)
=C_{q}(x).
\end{eqnarray*}
Now, we choose a  suitable $A$ such that $Ax=|x|e_{n}$,
where $e_{n}=(0,\ldots,0,1)\in\mathbb{S}^{n-1}$.
Then we have
\begin{eqnarray}\label{eq-2.4}
C_{q}(x)=C_{q}(|x|e_{n}).
\end{eqnarray}
Hence, to prove the lemma is sufficient to estimate the quality $C_{q}(\rho e_{n})$, where $\rho=|x|\in[0,1)$.

Using spherical coordinates transformation (cf. \cite[Equation (2.2)]{chen2018}), we get
\begin{eqnarray}\label{eq-2.5}
& &C_{q}(\rho e_{n})\\\nonumber
&= &\frac{1}{\omega_{n-1}}
 \int_{0}^{\pi} (1+\rho ^{2}- 2\rho\cos\theta_{1})^{(n-1)(q-1)}  \sin^{n-2}\theta_{1} d\theta_{1}\\\nonumber
 &&\times \int_{0}^{\pi} \sin^{n-3}\theta_{2}d\theta_{2}\cdots \int_{0}^{\pi} \sin \theta_{n-2}d\theta_{n-2} \int_{0}^{2\pi}d\theta_{n-1}\\\nonumber
&= &\frac{\omega_{n-2}}{\omega_{n-1}}
 \int_{0}^{\pi} (1+\rho ^{2}- 2\rho\cos\theta_{1})^{(n-1)(q-1)}  \sin^{n-2}\theta_{1} d\theta_{1}\\\nonumber
&=&(1+\rho)^{2(n-1)(q-1)}\frac{\omega_{n-2}}{\omega_{n-1}}
 \int_{0}^{\pi} \left(1- \frac{2\rho(1+\cos\theta_{1})}{(1+\rho)^{2}}\right)^{(n-1)(q-1)}  \sin^{n-2}\theta_{1} d\theta_{1},
\end{eqnarray}
where $\omega_{n-1}=\frac{2\pi^{n/2}}{\Gamma(n/2)}$ denotes the $(n-1)$-dimensional Lebesgue measure on $\mathbb{S}^{n-1}$ and
\begin{eqnarray}\label{eq-2.6}
\frac{\omega_{n-2}}{\omega_{n-1}}
=\frac{1}{ \int_{0}^{\pi} \sin^{n-2}\theta_{1}d\theta_{1}}
=\frac{\Gamma(\frac{n}{2})}{\sqrt{\pi}\Gamma(\frac{n-1}{2})}.
\end{eqnarray}
For any $\rho\in[0,1)$ and $\theta_{1}\in(0,\pi)$,
let
\begin{eqnarray}\label{eq-2.7}
s=\frac{4\rho}{(1+\rho)^{2}}\in[0,1)
\quad\text{and}\quad
t=\frac{1+\cos\theta_{1}}{2}\in(0,1).
\end{eqnarray}
By elementary calculations, we have that
\begin{eqnarray}\label{eq-2.8}
\sin\theta_{1}=2t^{\frac{1}{2}}(1-t)^{\frac{1}{2}}
\quad
\text{and}\quad
d\theta_{1}=-\frac{2}{\sin\theta_{1}} dt.
\end{eqnarray}
Then \eqref{eq-1.4}, \eqref{eq-1.6}, \eqref{eq-2.7} and \eqref{eq-2.8} imply that
\begin{eqnarray}\label{eq-2.9}
 &&\int_{0}^{\pi} \left(1- \frac{2\rho(1+\cos\theta_{1})}{(1+\rho)^{2}}\right)^{(n-1)(q-1)}  \sin^{n-2}\theta_{1} d\theta_{1}\\\nonumber
&=&
 2^{n-2}\int_{0}^{1} (1-st)^{(n-q)(q-1)}t^{\frac{n-3}{2}}(1-t)^{\frac{n-3}{2}}dt\\\nonumber
 &=&
 2^{n-2}\frac{\big(\Gamma(\frac{n-1}{2}) \big)^{2}}{\Gamma(n-1)}
  \;_{2}F_{1}\left(-(n-1)(q-1),\frac{n-1}{2};n-1;\frac{4\rho}{(1+\rho)^2} \right)\\\nonumber
  &=&
 2^{n-2}\frac{\big(\Gamma(\frac{n-1}{2}) \big)^{2}}{\Gamma(n-1)}
(1+\rho)^{-2(n-1)(q-1)}
\;_{2}F_{1}\left(-(n-1)(q-1),\frac{n}{2}+q-nq;\frac{n}{2};\rho^{2}\right).
\end{eqnarray}
Since $\sqrt{\pi}\Gamma(n-1)=2^{n-2}\Gamma(\frac{n-1}{2})\Gamma(\frac{n}{2})$ (cf. \cite[Page 23]{rain}), then by \eqref{eq-2.4}, \eqref{eq-2.5}, \eqref{eq-2.6} and \eqref{eq-2.9}, we get
$$
C_{q}(x)
=C_{q}(|x| e_{n})
=
\;_{2}F_{1}\left(-(n-1)(q-1),\frac{n}{2}+q-nq;\frac{n}{2};|x|^{2}\right),
$$
as required.
\end{proof}

\begin{lemma}\label{lem-2.3}
For any $q\in(1,\infty)$ and $x\in\mathbb{B}^{n}$,
we have
$$
\sup_{x\in\mathbb{B}^{n}}C_{q}(x)
=C_{q}(e_{n})=
\;_{2}F_{1}\left(-(n-1)(q-1),\frac{n}{2}+q-nq;\frac{n}{2};1\right).
$$
\end{lemma}

\begin{proof}
For any $q\in(1,\infty)$, since $\frac{n}{2}+(n-1)(q-1)-(\frac{n}{2}+q-nq)>0$, then the series $\;_{2}F_{1}\left(-(n-1)(q-1),\frac{n}{2}+q-nq;\frac{n}{2};\rho^{2}\right)$ is both absolutely convergent and continuous with respect to $\rho$ in $[0,1]$  (cf. \cite[Section 31]{rain}).
In the following, we divide the proof into two cases according to the value of $q$.
\begin{case} $q\in(1,1+\frac{1}{n-1})$.
\end{case}
In this case, we let
\begin{eqnarray}\label{eq-2.10}
\varphi(\rho)=\;_{2}F_{1}\left(-(n-1)(q-1),\frac{n}{2}+q-nq;\frac{n}{2};\rho\right) \end{eqnarray}
in $[0.1]$.
By Lemma \ref{lem-2.2}, we see that it suffices to prove  $\max_{\rho\in[0,1]}\varphi(\rho)=\varphi(1)$.

Using the formula \eqref{eq-1.5},
we get
\[\begin{split}
&\varphi'(\rho)\\
=& -\frac{2(n-1)(q-1)}{n}\left(\frac{n}{2}+q-nq\right)
\;_{2}F_{1}\left(1-(n-1)(q-1),\frac{n}{2}+q-nq+1;\frac{n}{2}+1;\rho\right)\\
=& -\frac{2(n-1)(q-1)}{n}\left(\frac{n}{2}+q-nq\right)
\;_{2}F_{1}\left(\frac{n}{2}+q-nq+1 ,1-(n-1)(q-1);\frac{n}{2}+1;\rho\right).
\end{split}\]
By the assumption $q\in(1,1+\frac{1}{n-1})$, we know that $$\frac{n}{2}+1>1-(n-1)(q-1)>0.$$
Then it follows from \eqref{eq-1.4} that
\begin{eqnarray*}
&&\;_{2}F_{1}
\left(\frac{n}{2}+q-nq+1 ,1-(n-1)(q-1);\frac{n}{2}+1;\rho\right)\\
&=&\frac{\Gamma(\frac{n}{2}+1)}{\Gamma\big(1-(n-1)(q-1)\big)\Gamma\big(\frac{n}{2}+(n-1)(q-1)\big)}
\int_{0}^{1}\frac{t^{-(n-1)(q-1)}(1-t)^{\frac{n}{2}+(n-1)(q-1)-1}}{(1-t\rho)^{\frac{n}{2}+q-nq+1}}dt\\
&>&0.
\end{eqnarray*}
Further, because
$\frac{n}{2}+q-nq<0$,
we obtain from these equalities that $\varphi'(\rho)\geq 0$ in $(0,1)$.
Hence,
$$
\max_{\rho\in[0,1]}\varphi(\rho)=\varphi(1).
$$
This, together with \eqref{eq-2.10} and Lemma \ref{lem-2.2},
 shows that
 $\sup_{x\in\mathbb{B}^{n}}C_{q}(x)
=C_{q}(e_{n})$
for any $q\in(1,1+\frac{1}{n-1})$, which is what we need.
\begin{case}
$q\in[1+\frac{1}{n-1},\infty)$.
\end{case}
For any $\rho\in(0,1)$,
it follows from \eqref{eq-2.5} that
\begin{eqnarray*}
& &\frac{\omega_{n-1}}{\omega_{n-2}}\cdot\frac{d}{d\rho}C_{q}(\rho e_{n})\\\nonumber
&= &2(n-1)(q-1)
 \int_{0}^{\pi} \sin^{n-2}\theta_{1} (\rho-\cos\theta_{1}) (1+\rho ^{2}- 2\rho\cos\theta_{1})^{(n-1)(q-1)-1}  d\theta_{1}\\
&= &2(n-1)(q-1)\rho
 \int_{0}^{\pi} \sin^{n-2}\theta_{1}  (1+\rho ^{2}- 2\rho\cos\theta_{1})^{(n-1)(q-1)-1}  d\theta_{1}\\
&&-2(n-1)(q-1)
 \int_{0}^{\pi} \cos\theta_{1} \sin^{n-2}\theta_{1} (1+\rho ^{2}- 2\rho\cos\theta_{1})^{(n-1)(q-1)-1}  d\theta_{1}.
\end{eqnarray*}
For the last integral above,
basic calculations yield that
\[\begin{split}
& \int_{0}^{\pi} \cos\theta_{1} \sin^{n-2}\theta_{1} (1+\rho ^{2}- 2\rho\cos\theta_{1})^{(n-1)(q-1)-1}  d\theta_{1}\\\nonumber
=& -\int_{-\frac{\pi}{2}}^{\frac{\pi}{2}}
\sin\theta_{1} \cos^{n-2}\theta_{1} (1+\rho ^{2}+ 2\rho\sin\theta_{1})^{(n-1)(q-1)-1}  d\theta_{1}\\
=& \int_{0}^{\frac{\pi}{2}}
\sin\theta_{1} \cos^{n-2}\theta_{1}
\big( (1+\rho ^{2}- 2\rho\sin\theta_{1})^{(n-1)(q-1)-1}
- (1+\rho ^{2}+ 2\rho\sin\theta_{1})^{(n-1)(q-1)-1}
\big) d\theta_{1}\\
\leq&\;0.
\end{split}\]
Then we obtain
$\frac{d}{d\rho}C_{q}(\rho e_{n})
\geq0$,
which means that
$\max_{\rho\in[0,1]}C_{q}(\rho e_{n})
=C_{q}(e_{n})$.
This, together with \eqref{eq-2.4}, shows that
$\sup_{x\in\mathbb{B}^{n}}C_{q}(x)
=\max_{\rho\in[0,1]}C_{q}(\rho e_{n})
=C_{q}(e_{n})$ for any $q\in[1+\frac{1}{n-1},\infty)$.
The proof of the lemma is complete.
\end{proof}
\subsection{Proofs of Theorems \ref{thm-1.1} and \ref{thm-1.2}}
For any $q\in(1,\infty)$ and $x\in\mathbb{B}^{n}$,
by Lemmas \ref{lem-2.1}$\sim$\ref{lem-2.3},
we see that
\begin{eqnarray*}
|u(x)|
 \leq
 \frac{  C^{\frac{1}{q}}_{q}(x)}{(1-|x|^2)^{\frac{(n-1)(q-1)}{q}}}
\|\phi\|_{L^{p}},
\end{eqnarray*}
where
\begin{eqnarray*}
C_{q}(x)
&=&
\;_{2}F_{1}\left(-(n-1)(q-1),\frac{n}{2}+q-nq;\frac{n}{2};|x|^{2}\right)\\
&\leq&
\;_{2}F_{1}\left(-(n-1)(q-1),\frac{n}{2}+q-nq;\frac{n}{2};1\right).
\end{eqnarray*}
The above inequalities are sharp.
 \qed

\begin{remark}
In the case $n=3$, we can find a very explicit sharp point estimate.
For any $q\in(1,\infty)$ and $\rho\in(0,1)$, it follows from \eqref{eq-2.5} and \eqref{eq-2.6} that
\begin{eqnarray*}
 C_{q}(\rho e_{3})
&= &\frac{1}{2}
 \int_{0}^{\pi} (1+\rho ^{2}- 2\rho\cos\theta_{1})^{2(q-1)}  \sin\theta_{1} d\theta_{1}\\
&= &\frac{1}{2}
 \int_{-1}^{1} (1+\rho ^{2}- 2\rho x)^{2(q-1)}   dx\\
&= &\frac{(1+\rho)^{4q-2}-(1-\rho)^{4q-2}}{4(2q-1)\rho}.
\end{eqnarray*}
If $\rho=0$, we obtain from  \eqref{eq-2.5} that $ C_{q}(0)=1$.
Therefore, for any $q\in(1,\infty)$ and $x\in\mathbb{B}^{n}\backslash\{0\}$, by Lemma \ref{lem-2.1} and \eqref{eq-2.3}$\sim$\eqref{eq-2.5},
we have the following sharp inequality
\begin{eqnarray*}
|u(x)|
 \leq
 \frac{1}{(1-|x|^2)^{\frac{2q-2}{q}}}
\left(\frac{(1+|x|)^{4q-2}-(1-|x|)^{4q-2}}{4(2q-1)|x|}\right)^{\frac{1}{q}}
\|\phi\|_{L^{p}}.
\end{eqnarray*}
For $q\in(1,\infty)$ and $x=0$, we have $|u(x)|\leq \|\phi\|_{L^{p}}$.
\end{remark}

\vspace*{5mm}
\noindent{\bf Funding.}
 The first author is partially supported by NSFS of China (No. 11571216, 11671127 and 11801166),
 NSF of Hunan Province (No. 2018JJ3327), China Scholarship Council and the construct program of the key discipline in Hunan Province.

\end{document}